\documentclass[11pt]{amsart}

\usepackage{amsmath}
\usepackage{amsfonts}
\usepackage{amssymb}
\usepackage{fullpage}
\usepackage{amscd}
\usepackage{amsthm}
\usepackage[numbers]{natbib} 
\usepackage{framed}
\usepackage{latexsym}

\usepackage{hyperref}
\hypersetup{colorlinks,linkcolor={red},citecolor={blue},urlcolor={blue}}

\newtheorem{theorem}{Theorem}[section]
\newtheorem{lemma}[theorem]{Lemma}

\newtheorem{corollary}[theorem]{Corollary}

\theoremstyle{definition}
\newtheorem{definition}[theorem]{Definition}

\numberwithin{equation}{section}

\newcommand{\CC}{\mathbb C}
\newcommand{\HH}{\mathbb H}

\newcommand{\cD}{\mathcal D}
\newcommand{\cA}{\mathcal A}

\newcommand{\PP}{\mathbb P}
\newcommand{\QQ}{\mathbb Q}
\newcommand{\RR}{\mathbb R}
\newcommand{\ZZ}{\mathbb Z}

\newcommand{\SL}{\mathop{\mathrm {SL}}\nolimits}

\newcommand{\Orth}{\mathop{\null\mathrm {O}}\nolimits}

\newcommand{\II}{\mathop{\mathrm {II}}\nolimits}
\def\Borch{\mathbf{B}}
\begin{document}

\title[On the non-existence of singular Borcherds products]{On the non-existence of singular Borcherds products}

\author{Haowu Wang}

\address{School of Mathematics and Statistics, Wuhan University, Wuhan 430072, Hubei, China}

\email{haowu.wangmath@whu.edu.cn}

\author{Brandon Williams}

\address{Lehrstuhl A für Mathematik, RWTH Aachen, 52056 Aachen, Germany}

\curraddr{Institute of Mathematics, University of Heidelberg, 69120 Heidelberg, Germany}

\email{bwilliams@mathi.uni-heidelberg.de}

\subjclass[2020]{11F22, 11F27, 11F55}

\date{\today}

\keywords{Automorphic products, rational quadratic divisors, orthogonal modular forms, singular weight, Borcherds--Kac--Moody algebras}

\begin{abstract} 
Let $l\geq 3$ and $F$ be a modular form of weight $l/2-1$ on $\mathrm{O}(l,2)$ which vanishes only on rational quadratic divisors. We prove that $F$ has only simple zeros and that $F$ is anti-invariant under every reflection fixing a quadratic divisor in the zeros of $F$. In particular, $F$ is a reflective modular form. As a corollary, the existence of $F$ leads to $l\leq 20$ or $l=26$, in which case $F$ equals the Borcherds form on $\II_{26,2}$. This answers a question posed by Borcherds in 1995. 
\end{abstract}

\maketitle

\section{Introduction}

In this paper we prove some nice properties of holomorphic automorphic products of singular weight on orthogonal groups $\Orth(l,2)$, and resolve an open problem posed by Borcherds in 1995.  

Let $M$ be an even integral lattice of signature $(l,2)$ with $l\geq 3$. Let $(-,-)$ denote the bilinear form on $M$ and let $M'$ be the dual lattice of $M$. Either of the two connected components of the space
\begin{equation*}
\{[\mathcal{Z}] \in  \PP(M\otimes \CC):  (\mathcal{Z}, \mathcal{Z})=0, (\mathcal{Z},\bar{\mathcal{Z}}) < 0\}
\end{equation*}
defines the attached Hermitian symmetric domain $\cD(M)$; we fix a component once and for all.  
We denote by $\Orth^+ (M)$ the subgroup of the orthogonal group of $M$ that preserves $\cD(M)$. Let $k$ be an integer, $\Gamma$ be a finite-index subgroup of $\Orth^+(M)$ and $\chi$ be a character of $\Gamma$. A \textit{modular form} of weight $k$ and character $\chi$ for $\Gamma$ is a holomorphic function $F$ on the affine cone $$\cA(M) = \{\mathcal{Z} \in \PP(M\otimes \CC): [\mathcal{Z}] \in \cD(M)\}$$ over $\cD(M)$ which satisfies
\begin{align*}
F(t\mathcal{Z})&=t^{-k}F(\mathcal{Z}), \quad \forall t \in \CC^\times,\\
F(g\mathcal{Z})&=\chi(g)F(\mathcal{Z}), \quad \forall g\in \Gamma.
\end{align*}
The symmetric domain $\cD(M)$ can be realized as a tube domain around any cusp. The action of $\Orth^+(M)$ on the tube domain induces an automorphy factor, with respect to which one can also define modular forms of half-integral weight. If $F$ is nonzero, then either $k=0$ (in which case $F$ is constant), or $k\geq l/2-1$. The smallest possible weight $l/2-1$ of a non-constant modular form is called the \textit{singular} weight. On any tube domain realization, modular forms can be expanded into Fourier series, and singular-weight modular forms are characterized by the fact that their Fourier series is supported only on norm-zero vectors.

In 1995 and 1998 Borcherds \cite{Bor95, Bor98} established a remarkable lift to construct orthogonal modular forms with nice properties. The Borcherds lift is a multiplicative map and its input $f$ is a weakly holomorphic modular form of weight $1-l/2$ for the Weil representation of $\SL_2(\ZZ)$ attached to $M'/M$. The image $\Borch(f)$ is a meromorphic modular form for the discriminant kernel
$$
\widetilde{\Orth}^+(M) = \{ g \in \Orth^+(M) : g(v)-v \in M, \quad \text{for all $v\in M'$} \}
$$
whose divisor is a linear combination of rational quadratic divisors
$$
\lambda^\perp = \{ [\mathcal{Z}] \in \cD(M) : (\mathcal{Z}, \lambda)=0 \}, \quad \text{for $\lambda\in M$ with  $\lambda^2>0$.}
$$
The function $\Borch(f)$ has an infinite product expansion around any cusp, so we call it a \textit{Borcherds product} or an \textit{automorphic product}.   

Holomorphic Borcherds products of singular weight are very exceptional. They are usually the denominators of generalized Kac--Moody algebras that have a natural construction as the BRST cohomology related to a superconformal field theory (e.g. \cite{Bor90,Bor92, Sch00}). For example, the fake monster Lie algebra, which was constructed by Borcherds \cite{Bor90} in 1990, is the BRST cohomology related to the Leech lattice vertex operator algebra. This infinite dimensional Lie algebra describes the physical states of a bosonic string moving on a 26-dimensional torus, and its denominator is given by the singular-weight Borcherds product $\Phi_{12}$ on $\II_{26,2}$, the even unimodular lattice of signature $(26,2)$. 
Motivated by this connection, in 1995 Borcherds \cite[Open problem 3 in Section 17]{Bor95} posed the following problems:

\vspace{2mm}

\begin{enumerate}
    \item Are there a finite or infinite number of singular-weight modular forms which can be written as automorphic products?
    \item Are there holomorphic automorphic products of singular weight on $\Orth(l,2)$ when $l>26$? 
\end{enumerate}

\vspace{2mm}

These problems have remained open, although there are some partial results in the literature. Dittmann, Hagemeier and Schwagenscheidt \cite{DHS15} and Opitz and Schwagenscheidt \cite{OS19} classified holomorphic Borcherds products of singular weight on simple lattices (i.e. lattices on which there is no obstruction to constructing Borcherds products). Note that there are only finitely many simple lattices and a full classification was given in \cite{BEF16}. Scheithauer \cite{Sch17} classified singular Borcherds products on unimodular lattices, and derived a conditional bound on the signature of lattices of prime level with singular Borcherds products. 

This paper gives a complete solution to problem (2) and a weak solution to problem (1). To state the main results more conveniently, we introduce the following definition.

\begin{definition}
A non-constant modular form for $\Gamma<\Orth^+(M)$ is called \textit{special} if its zero divisor is a linear combination of quadratic divisors $\lambda^\perp$ for $\lambda\in M\otimes\QQ$ with $\lambda^2>0$. 
\end{definition}

Bruinier's converse theorem \cite{Bru02, Bru14} shows that every special modular form for the discriminant kernel of $M$ is a Borcherds product on $M$ if $M$ splits as $U\oplus U(m)\oplus L$, where $U$ is the unique even unimodular lattice of signature $(1,1)$. It is not known whether this holds for arbitrary $M$ (of signature $(l, 2)$ with $l \ge 3$).

In this paper, we first describe zeros of special modular forms of singular weight. The proof follows from studying the Laplace operator on the tube domain around any cusp. 

\begin{theorem}\label{MTH1}
Let $M$ be an even lattice of signature $(l,2)$ with $l\geq 3$ and $F$ be a special modular form of singular weight for a finite-index subgroup of $\Orth^+(M)$. Then $F$ has only simple zeros. Moreover, if $F$ vanishes on the quadratic divisor $\lambda^\perp$, then we have 
$$
F(\sigma_\lambda(\mathcal{Z})) = - F(\mathcal{Z}),
$$
where $\sigma_\lambda$ is the reflection fixing $\lambda^\perp$ defined as
$$
\sigma_\lambda(v) = v -\frac{2(\lambda,v)}{(\lambda,\lambda)}\lambda, \quad v\in M\otimes\QQ.  
$$
\end{theorem}

Following Borcherds \cite{Bor98} and Gritsenko--Nikulin \cite{GN98}, we define reflective modular forms. 

\begin{definition}
Let $F$ be a special modular form for a finite-index subgroup of $\Orth^+(M)$. If $\sigma_\lambda$ lies in $\Orth^+(M)$ whenever $F$ vanishes on $\lambda^\perp$, then $F$ is called \textit{reflective}. 
\end{definition}

As a corollary of Theorem \ref{MTH1}, we prove the following result, which is a conjecture formulated by the first named author in 2019 (see \cite[Conjecture 9.8]{Wan19}). 

\begin{theorem}\label{MTH2}
Let $M$ be an even lattice of signature $(l,2)$ with $l\geq 3$ and $F$ be a special modular form of singular weight for a finite-index subgroup $\Gamma$ of $\Orth^+(M)$. Then there exists an even lattice $\mathbb{M}$ on $M\otimes\QQ$ such that $\Gamma < \Orth^+(\mathbb{M})$ and $F$ is a reflective modular form on $\mathbb{M}$. 
\end{theorem}

Theorem \ref{MTH2} reduces the difficult problem of classifying singular automorphic products to the easier problem of classifying reflective modular forms. A full classification of reflective modular forms is not yet available, but many partial results have been obtained over the past two decades \cite{GN98, GN02, Bar03, Sch06, Sch17, Ma17, Ma18, Dit19, Wan18, GN18, Wan19, Wan22, Wan23a, Wan23b}. In particular, the first named author \cite{Wan18, Wan19, Wan23b} classified lattices of large rank which has a reflective modular form. This yields a complete solution to Borcherds' problem (2).

\begin{theorem}\label{MTH3}
There is no special modular form of singular weight on $\Orth(l,2)$ when $l\geq 21$ and $l\neq 26$. The Borcherds form $\Phi_{12}$ is the unique special modular form of singular weight on $\Orth(26,2)$ up to a constant factor.     
\end{theorem}

Special modular forms of singular weight on $\Orth(l,2)$ do exist for $l=3$, $4$, $6$, $8$, $10$, $12$, $14$, $18$. They do not appear to exist on $\Orth(l,2)$ for any other $l\leq 20$. 

The finiteness of lattices with reflective modular forms is known in certain specific cases \cite{Ma17, Ma18, Wan23b}. However, the finiteness of reflective modular forms is unknown in general. Therefore, we cannot give a full solution to Borcherds' problem (1). In 2018 Ma \cite[Corollary 1.10]{Ma18} proved that up to scaling there are only finitely many lattices of signature $(l,2)$ with $l\geq 4$ which has a reflective modular form with simple zeros.  Theorem \ref{MTH1}, Theorem \ref{MTH2} and Ma's result together yield the following weak solution to Borcherds' problem (1).

\begin{corollary}
The set of lattices $\mathbb{M}$ in Theorem \ref{MTH2} is finite up to scaling when $l\geq 4$.   
\end{corollary}

This paper is organized as follows. In Section \ref{sec:th1} we introduce the Laplace operator and prove Theorem \ref{MTH1}. In Section \ref{sec:th2} we prove Theorem \ref{MTH2} and give a corollary. Section \ref{sec:th3} is devoted to the proof of Theorem \ref{MTH3}.

\section{Proof of Theorem \ref{MTH1}}\label{sec:th1}
In this section we use the Laplace operator on a tube domain to prove Theorem \ref{MTH1}. 

Let $M$ be an even lattice of signature $(l,2)$ with $l\geq 3$. Let $c\in M$ be a primitive norm-zero vector and $c'\in M'$ with $(c,c')=1$. 
Let $e_1,...,e_{l}$ be any $\RR$-basis of the signature $(l-1,1)$ lattice
$$
M_{c,c'}=\{ x \in M : (x,c)=(x,c')=0 \}
$$ 
and define the Gram matrix $S = (s_{ij})_{i, j=1}^{l}$ with inverse $S^{-1} = (s^{ij})_{i, j}$, where $s_{ij} := (e_i, e_j)$.
The \textit{holomorphic Laplace operator} is defined as
$$
\mathbf{\Delta} = \mathbf{\Delta}_{c, c'} = \frac{1}{2} \sum_{i,j=1}^{l} s^{ij} \frac{\partial^2}{\partial e_i \partial e_j}.
$$
This operator acts on functions defined on the tube domain $\mathbb{H}_{c, c'}$ which is a connected component of the space
$$
\{ Z = X + iY :\; X, Y \in M_{c,c'}\otimes\RR, \; (Y,Y)<0  \}. 
$$
It is clear that $\mathbf{\Delta}_{c, c'}$ is independent of the basis $e_i$.

For any $\lambda \in M_{c,c'}\otimes\QQ$, we have 
$$
\mathbf{\Delta} e^{2\pi i (\lambda, Z)} = -2\pi^2 (\lambda,\lambda) e^{2\pi i (\lambda, Z)}.
$$ 
Therefore, a function $F$ that is represented near $c$ by a Fourier series $$F(Z) = \sum_{\lambda\in M_{c,c'}\otimes\QQ} c(\lambda) e^{2\pi i (\lambda, Z)}$$ satisfies $\mathbf{\Delta} F = 0$ if and only if it is \emph{singular at $c$}, that is, $c(\lambda) = 0$ for every $\lambda$ of nonzero norm.

For any meromorphic function $F : \mathbb{H}_{c, c'} \rightarrow \mathbb{C}$, matrix elements $g \in \mathrm{O}^+(M \otimes \mathbb{R})$ and positive integer $m $, we have (see e.g. \cite[Lemma 2.4]{Wil21})
$$
(\mathbf{\Delta}^m F) \Big|_{l/2 + m} g = \mathbf{\Delta}^m \Big( F \Big|_{l/2 - m} g \Big).$$ 
In particular, if $F$ transforms like a modular form of weight $l/2-m$ then $\mathbf{\Delta}^m F$ transforms like a modular form of weight $l/2 + m$. If $F$ is in fact a non-constant (holomorphic) modular form, then $m$ must be 1, i.e. $F$ is of singular weight, and thus $\mathbf{\Delta} F = 0$.

Theorem \ref{MTH1} is a particular case of the following result.

\begin{theorem}\label{th:divisor}
Let $F$ be a modular form of singular weight on $\Gamma<\Orth^+(M)$. 
Let $v$ be a positive-norm vector of $M$. If $F$ vanishes on the quadratic divisor $v^\perp$, then $v^\perp$ has multiplicity one in the divisor of $F$ and $F$ is anti-invariant under the associated reflection, i.e.
$$
F(\sigma_v(\mathcal{Z})) = -F(\mathcal{Z}). 
$$
\end{theorem}
\begin{proof}
First, note that the existence of a singular-weight modular form $F$ implies that $M$ contains an isotropic plane, i.e. $\cD(M)$ contains a one-dimensional cusp. (Indeed, $M$ has rank at least $5$ and therefore contains a primitive isotropic vector $c$. The Fourier expansion of $F$ about $c$ is supported only on norm-zero vectors orthogonal to $c$, and as $F$ is non-constant, such vectors must exist.)

Fix a primitive isotropic vector $c\in M$ and a vector $c'\in M'$ with $(c,c')=1$ such that $v\in M_{c,c'}$. We view $F$ as a modular form on the associated tube domain $\HH_{c,c'}$. Let $K$ be the orthogonal complement of $v$ in $M_{c,c'}$. We decompose $Z\in M_{c,c'}\otimes\CC$ as $Z=z_1v+z_2$ for $z_1\in \CC $ and $z_2\in K\otimes\CC$. 

If $d$ is the multiplicity of $v^\perp$ in the divisor of $F$, then the Taylor series of $F$ at $z_1=0$ has the form 
$$
F(Z)=f_d(z_2)z_1^d + O(z_1^{d+1}), \quad f_d(z_2)\neq 0. 
$$
By applying the Laplace operator on the tube domain $\HH_{c,c'}$ to $F$, we have
$$
0=\mathbf{\Delta}_{c,c'}(F)=\frac{d(d-1)}{2}f_d(z_2)z_1^{d-2}+O(z_1^{d-1}),
$$
which shows that $d=1$. Similarly, applying the Laplace operator to the function
$$
G(Z):=F(Z)+F|\sigma_v(Z) = F(z_1v+z_2)+F(-z_1v+z_2) = O(z_1^2)
$$
we find that $\mathbf{\Delta}(G)=0$, which forces $F(Z)+F|\sigma_v(Z)=0$.
\end{proof}

\section{Proof of Theorem \ref{MTH2}}\label{sec:th2}
We divide the proof of Theorem \ref{MTH2} into two lemmas. The notation is the same as Theorem \ref{MTH2}. 

\begin{lemma}\label{lem:arithmetic}
Let $\Gamma_F$ be the group generated by $\Gamma$ and the reflections $\sigma_r$ for which $F$ vanishes on $r^\perp$. Then $\Gamma_F$ is arithmetic in $\Orth^+(M\otimes\QQ)$. In particular, $\Gamma_F\cap \Orth^+(M)$ has finite index in $\Gamma_F$. 
\end{lemma}
\begin{proof}
Let $m$ be the order of the character of $F^2$ on $\Gamma$. Then $F^{2m}|g=F^{2m}$ for all $g\in \Gamma_F$ by Theorem \ref{th:divisor}. It follows that $\Gamma_F$ is a discrete subgroup of $\Orth^+(M\otimes\RR)$. Since $\cD(M) / \Gamma$ has finite volume, the smaller quotient space $\cD(M)/\Gamma_F$ also has finite volume. Therefore, $\Gamma_F$ is a lattice in $\Orth^+(M\otimes\QQ)$ and thus an arithmetic subgroup by the Margulis arithmeticity theorem \cite{Mar91}. 
\end{proof}

\begin{lemma}\label{lem:lattice}
The lattice $\mathbb{M}$ in Theorem \ref{MTH2} can be obtained by rescaling the rational lattice generated by $g(M)$, where $g$ runs over representatives of $\Gamma_F / (\Gamma_F\cap \Orth^+(M))$. 
\end{lemma}
\begin{proof}
Let $M_F$ be the lattice generated by $g(M)$ over $\ZZ$. Since $g\in \Orth^+(M\otimes\QQ)$, the lattice $M_F$ is rational. By construction, $\Gamma_F$ fixes $M_F$. Lemma \ref{lem:arithmetic} shows that the set of $g$ is finite.  Therefore, there exists $d \in \mathbb{N}$ such that the rescaled lattice $M_F(d)$ is integral. Since $\Orth^+(M_F)=\Orth^+(M_F(d))$, we can conclude that $M_F(d)$ is the desired $\mathbb{M}$.
\end{proof}

For lattices that split a unimodular plane, 
Theorem \ref{MTH2} sharply restricts the principal parts of vector-valued modular forms that lift to holomorphic Borcherds products
of singular weight:

\begin{corollary}
Let $F$ be a holomorphic Borcherds product of singular weight on a lattice of type $M=U\oplus K$. We denote by $\mathrm{m}_K$ the positive generator of the integral ideal generated by $v^2/2$ for $v\in K$. 
We write the principal part of the input of $F$ as
$$
\sum_{n>0} \sum_{x\in K'/K} c(x,-n)q^{-n} \mathbf{e}_x.  
$$
If $c(x,-n)\neq 0$, then $\mathrm{m}_K / n \in \ZZ$. 
\end{corollary}

\begin{proof}
Write vectors $\lambda\in U\oplus K'$ in the form $(a,r,b)$ with $a,b\in \ZZ$ and $v\in K'$, such that $\lambda^2=r^2-2ab$. Suppose $c(x,-n)\neq 0$, and let $d$ be the largest integer such that $c(dx,-d^2n)\neq 0$. For simplicity we set $y=dx \in K'/K$, $m=d^2n$, and $v_r:=(1,r,r^2/2-m)$ for any $r\in y+K$. Then $v_r^\perp$ has multiplicity $c(y,-m)$ in the divisor of $F$, so $c(y,-m)$ is a positive integer. 

Let $t$ be the minimal positive integer such that $tv_r$ lies in the lattice $M_F$ generated by $g(M)$, where $g$ runs through the representatives of $\Gamma_F / (\Gamma_F\cap \Orth^+(M))$ for $\Gamma=\widetilde{\Orth}^+(M)$ (as in Lemmas \ref{lem:arithmetic} and \ref{lem:lattice}). Since all vectors of type $v_r$ lie in a single $\widetilde{\Orth}^+(M)$-orbit, the integer $t$ depends only on the class $y\in K'/K$ and the number $m$. For any $r,s\in y+K$, we have $\sigma_{v_r}(tv_s) \in M_F$, which is equivalent to 
$$
\frac{(t v_s, v_r)}{m} v_r = \frac{(v_s, v_r)}{m} \cdot tv_r \in M_F.
$$
It follows that $(v_s, v_r)/m\in \ZZ$ for all $r,s\in y+K$. Therefore,  $\mathrm{m}_K/m \in \ZZ$ and $\mathrm{m}_K / (nd^2) \in \ZZ$. 
\end{proof}

\section{A proof of Theorem \ref{MTH3}}\label{sec:th3}
In this section we prove Theorem \ref{MTH3} as a corollary of Theorem \ref{MTH2} using the following results:

\begin{theorem}[Theorem 1.2 of \cite{Wan23b}]\label{th:wang}
\noindent
\begin{enumerate}
    \item There is no even lattice of signature $(l,2)$ with a reflective modular form when $l=21$ or $l\geq 23$ and $l\neq 26$.
    \item Let $M=U\oplus K$ be an even lattice of signature $(l,2)$ which has a reflective Borcherds product. If $l=26$, then $M\cong \II_{26,2}$ and if $l=22$, then $M\cong 2U\oplus D_{20}$. 
\end{enumerate}    
\end{theorem}

Let $M$ be an even lattice of signature $(l,2)$ with $l\geq 21$ and suppose $F$ is a special modular form of singular weight for some finite-index subgroup $\Gamma$ of $\Orth^+(M)$. By Theorem \ref{MTH2}, Lemma \ref{lem:arithmetic} and Lemma \ref{lem:lattice}, $F$ is a reflective modular form for $\Gamma_F<\Orth^+(\mathbb{M})$. Theorem \ref{th:wang} implies that $l=22$ or $26$.  By \cite[Corollory 3.2]{Ma18}, there exists a lattice $\mathbb{M}_1$ in $\mathbb{M}\otimes\QQ$ such that $\Orth^+(\mathbb{M}) \subset \Orth^+(\mathbb{M}_1)$ and such that $\mathbb{M}_1$ is a scaling of an even lattice of type $\mathbb{M}_2=2U\oplus L$. The subgroup $\Gamma_F$ has finite index in $\Orth^+(\mathbb{M}_2)$, and we denote the index by $a$. Then the function
$$
\hat{F} = \prod_{g \in \Orth^+(\mathbb{M}_2) / \Gamma_F} F|g
$$
defines a reflective modular form of weight $a(l/2-1)$ for $\Orth^+(\mathbb{M}_2)$. 

When $l=26$, we conclude from Theorem \ref{th:wang} that $\mathbb{M}_2 \cong \II_{26,2}$ and $\hat{F}=\Phi_{12}^a$ up to a constant factor. By Theorem \ref{MTH1}, $F$ has only simple zeros. Therefore, $\Phi_{12} / F$ is a (holomorphic) modular form of weight $0$ for $\Gamma_F$ and thus constant. 

When $l=22$, Theorem \ref{th:wang} implies that $\mathbb{M}_2 \cong 2U\oplus D_{20}$. Borcherds \cite{Bor00} constructed a reflective modular form $\Psi_{24}$ of weight $24$ for $\Orth^+(2U\oplus D_{20})$ which vanishes on $v^\perp$ with multiplicity $1$ and vanishes on $u^\perp$ with multiplicity $8$, where $v\in 2U\oplus D_{20}$ with $v^2=2$ and $u\in 2U\oplus D_{20}'$ with $u^2=1$. We know from \cite[Lemma 3.2]{Wan23b} that $\Psi_{24}$ is the unique reflective modular form for $\Orth^+(2U\oplus D_{20})$ up to a power. It follows that $\hat{F}=\Psi_{24}^d$ for some positive integer $d$. Recall that $F$ has only simple zeros. By comparing the weights and zero orders of $F$ and $\Psi_{24}$, we find
$$
10a = 24 d \quad \text{and} \quad 8d \leq a,
$$
which leads to $10a/3 \leq a$, a contradiction.  This finishes the proof of Theorem \ref{MTH3}. 

\bibliographystyle{plainnat}
\bibliofont
\bibliography{refs}

\end{document}